\documentclass{amsart}
\usepackage{amsfonts}
\usepackage{amssymb}
\usepackage{graphicx}
\usepackage{amsmath}

\newtheorem{theorem}{Theorem}

\newtheorem{definition}[theorem]{Definition}

\newtheorem{lemma}[theorem]{Lemma}

\newtheorem{proposition}[theorem]{Proposition}
\newtheorem{remark}[theorem]{Remark}

\newcommand{\pie}{\Pi E}

\begin{document}
	\title{Decomposition of matrices
	into product of idempotents and separativity of regular rings}
	\author{S. K. Jain and A. Leroy}
	\address{S. K. Jain\\
		Department of Mathematics\\
		Ohio University, USA\\
		Email:jain@ohio.edu\\
		Andre Leroy\\
		Universit\'e d'Artois\\
		Lens, France\\
		Email:andre.leroy@univ-artois.fr}
	

\maketitle

\centerline{\it Dedicated to Abraham Berman on his $80^{th}$ birthday.}
	
\begin{abstract}
	Following O'Meara's result [Journal of Algebra and Its Applications Vol~\textbf{13}, No. 8 (2014)], it follows that the block matrix 
	$
	A=\begin{pmatrix}
		B & 0 \\
		0 & 0
	\end{pmatrix} \in M_{n+r}(R)
	$, $B\in M_n(R)$, $r\ge 1$, 
	over a von Neumann regular separative ring $R$, 
	is a product of idempotent matrices. Furthermore, this decomposition into idempotents of $A$ also holds when $B$ is an invertible matrix and $R$ is a GE ring (defined by Cohn [New mathematical monographs: {\bf 3}, Cambridge University Press (2006)]). As a consequence, it follows that if there exists an example of a von Neumann regular ring $R$ over which the matrix
	$
	A=\begin{pmatrix}
		B & 0 \\
		0 & 0
	\end{pmatrix} \in M_{n+r}(R)
	$
	where $B\in M_n(R)$, $r\ge 1$  
	, cannot be expressed as a product of idempotents,
	then $R$ is not separative, thus providing an answer to an open question whether there exists a von Neumann regular ring which is not separative.
	The paper concludes with an example of an open question whether every totally nonnegative matrix is a product of nonnegative idempotent matrices.
	\end{abstract}

	\vspace{0.4cm}

	
	\section{Introduction and notation}
	Decompositions of singular matrices over a field or division ring (cf. J. Erdos, Laffey) as product of idempotent matrices led to the study of the decomposition of certain elements of a ring as a product of idempotents. 
	There is a plenty of literature related to this question, where  some of these are concerned with semigroups, nonegative matrices, or von Neumann regular rings (cf.  Gould, Jain-Leroy, O'Meara).
All rings are unital.   $U(R)$, $E(R)$, and $\pie (R)$ denote the set of units, the set of idempotents and the set of elements in $R$ that are product of idempotents, respectively.

We now give some basic facts about the decompositions of zero divisors of a ring into products of idempotents. 

If an element $a$ of a ring can be written as product of idempotents then both its left and right annihilators are nonzero.   This leads to the study of the following property
$$
\forall a\in R,\;  l(a)\ne 0  \Leftrightarrow r(a)\ne 0  \quad \quad\quad \quad  (*)
$$
We note that unit regular rings, artinian rings, matrix rings over quasi euclidian rings, amongst others, have this property (*) (cf. \cite{FL})

For example, for any ring $S$, consider the matrix 
$A=\begin{pmatrix}
	x & y \\
	0 & 0
\end{pmatrix}\in R=M_2(S)$.   Clearly $l(A)\ne 0$ and, if $A$ is a product of idempotent matrices, we must have $r(A)\ne0$.  Of course, this condition is not sufficient to write $A$ as a product of idempotents.   For instance, if $S=\mathbb Q[X,Y]$ and $X=x, Y=y$ we have $XY \in XS\cap YS$.   But it is well known that, in this case, the matrix $A$ is not a product of idempotent matrices (\cite{bb}). 

We collect some basic facts around matrices over any ring with a zero row.  These matrices will be the main topic of the next section.  For any matrix $A\in M_n(R)$ we write $A_i=(1-e_{ii})A$, the matrix obtained from $A$ by replacing its $i^{th}$ row by a zero row.

\section{Preliminaries}

\begin{lemma}
\label{first remarks}
	\begin{enumerate}
		\item Suppose $u\in U(R)$, $e\in E(R)$, and $r\in eR$ are such that $eu, u^{-1}r \in \pie (R)$ then  $r, ru, ur \in \pie (R)$.  
		\item Let $A\in M_n(R)$ with rows 
		$L_1,\dots , L_n$ be such that 
		$L_i=\Sigma_{j\ne i} \alpha_j L_j$ with $1-\Sigma_{j\ne i} \alpha_j\in U(R)$.   Then if $A_i=(1-e_{ii})A$ is a product of idempotents, the same is true for $A$.
		\item Let $R$ be a projective free ring then any matrix $A\in M_n(R)$ which is a product of proper idempotent matrices is similar to a matrix with a zero row (column).
	\end{enumerate}
\end{lemma} 
\begin{proof}
(1) We have $r=er=euu^{-1}r$.  Hence our hypothesis shows that $r \in \pie$.  We also have $u^{-1}ru \in \pie$ and hence $ru=eru=euu^{-1}ru\in \pie$.  Finally $ur=u(ru)u^{-1}\pie$.

(2) Let $u^{-1}=\Sigma_{j\ne i} e_{jj} + (1-\Sigma_{j\ne i}\alpha_j)e_{ii}\in Gl_n(R)$.  We then have have  $u^{-1}A=A_i=(1-e_{ii})A$ and we check that the result follows by applying (1).

(3) If $A$ is a product of idempotent matrices then there exists $E=E^2\in M_n(R)$ such that $A=EA$.   Since $R$ is projective free there exists an invertible matrix $P\in M_n(R)$ such that $PEP^{-1}$ is a diagonal matrix with zero and one on the diagonal. 
We thus have $PAP^{-1}= PEP^{-1}PAP^{-1}$ and hence $PAP^{-1}$ has a zero row. 
\end{proof}

The first statement of Lemma \ref{first remarks} shows that it is important to know when an invertible element $u\in U(R)$ and an idempotent element $e=e^2$ are such that $eu \in \pie $.  We will give two examples of such a behavior in the frame of matrix rings.  Remark first that a matrix in $M_n(R)$ with its $i^{th}$ row zero belongs to $(1-e_{ii})M_n(R)$. 
Let us recall that a matrix $A\in M_n(R)$ is a permutation matrix if it each row and each column has only one nonzero entry equal to $1$. 
\begin{definition}
A matrix $A\in M_n(R)$ is a quasi permutation matrix (resp. quasi elementary matrix) if there exists $1\le i \le n$ and a permutation matrix $P$ (resp. elementary matrix $Q$) such that either $A=(1-e_{ii})P$ or $A=P(1-e_{ii})$ (resp $A=(1-e_{ii})Q$ or  $A=Q(1-e_{ii})$).
\end{definition}

It was proved in \cite{AJL1} that quasi permutation matrices are products of idempotent matrices.  Here we prove the similar result for quasi elementary matrices.

\begin{proposition}
	\label{Quasi permutation and quasi elementary}
	A quasi elementary matrix is a product of idempotent matrices. 
\end{proposition}
\begin{proof}
	 Assume that $A$ is a quasi elementary matrix.   We note that a quasi elementary matrix always has a zero row and a zero column.  Such a matrix is always similar to a matrix having last row zero.  In other words, we may assume that our quasi elementary matrix $A$ is of the form $A=(I_{n}-e_{nn})Q$ where $Q$ is an elementary matrix of the form $Q=I_{n}+be_{ij}$, where $1\le i \ne j\le n$.   If $j=n$, then $A=(I_{n}-e_{nn})(I_{n}+be_{in})=I_{n}+be_{in}-e_{nn}$.  This matrix is easily seen to be idempotent.  If $j\ne n$, we have   
	$$
	A=\begin{pmatrix}
		I_{n-1} + be_{ij} & 0 \\
		0 & 0
	\end{pmatrix}=\begin{pmatrix}
		I_{n-1} & 0 \\
		0 & 0
	\end{pmatrix}\begin{pmatrix}
		I_{n-1}+ be_{ij} & e_i^T \\
		-be_j & 0
	\end{pmatrix}\begin{pmatrix}
		I_{n-1} & 0 \\
		0 & 0
	\end{pmatrix}
	$$
	We can easily check that all the three factors are idempotent matrices. 
\end{proof}

\begin{proposition}
\label{permutation and elementary matrices acting on a zero row matrix}
	Let $A\in M_n(R)$ be a matrix with a zero row (resp. column) and $Q$ be an elementary matrix or a permutation matrix.  Then 
	$A$ is a product of idempotent matrices if and only if $QA$ (resp. $AQ$) is a product of idempotent matrices.   
\end{proposition}
\begin{proof}
	First suppose that the matrix $A$ has its $i^{th}$ row zero and is a product of idempotent matrices.
%
	Let $Q=Q_{k,l}=Id +ae_{kl}$ with $a\in R$ be an elementary matrix.  If $i\notin \{k,l\}$ then comparing the rows of the matrices on both sides we get that $Q_{k,l}A=(id -e_{ii})Q_{k,l}A$.  If $i=k$, we have $Q_{i,l}A=
	(Id.-e_{ii} + ae_{il})A$ (remark that $Id.-e_{ii} + ae_{il}$ is an idempotent).  Finally if $i=l$ we have $Q_{k,i}A=A$ .  In the three cases we conclude that if $A$ is a product of idempotents then the same is true for $Q_{k,l}A$.  Conversely, noting that
	$Q^{-1}$ is also an elementary matrix, we know, by Proposition \ref{Quasi permutation and quasi elementary}, that the quasi elementary matrix $(1-e_{ii})Q^{-1}$ is a product of idempotent matrices.  Since the $i^{th}$ row of $A$ is a zero row, we have 
	$A=(Id-e_{ii})A=(1-e_{ii})Q^{-1}QA$.  This shows that if $QA$ is a 
	product of idempotent matrices, the same is true for $A$.  
	
	We leave the proof of the case when $Q$ is a permutation matrix to the reader. 
\end{proof}

\section{Main results}

We now consider a matrix $P\in M_{n+1}(R)$ 
of the form $\begin{pmatrix}
B & C \\
0 & 0
\end{pmatrix}$.
\begin{theorem}
\label{3 statements}
	\begin{enumerate}
		\item If $B\in M_n(R)$ is a product of idempotent matrices then so is $\begin{pmatrix}
			B & C \\
			0 & 0
		\end{pmatrix}$, for any column $C$.
		\item If $B\in M_n(R)$ is invertible then $\begin{pmatrix}
			B & C \\
			0 & 0
		\end{pmatrix}$ is similar to the matrix $\begin{pmatrix}
			B & 0 \\
			0 & 0
		\end{pmatrix}$.
	\item If the matrix $\begin{pmatrix}
		B & 0 \\
		0 & 0
	\end{pmatrix}$ is a product of idempotents then 
 so is the matrix $\begin{pmatrix}
 	B & C \\
 	0 & 0
 \end{pmatrix}$ where $C=BQ$ for some column $Q\in M_{n1}(R)$
	\end{enumerate}
\end{theorem}

\begin{proof}

(1) $B$ is a product of idempotent matrices, say $B=E_1\cdots E_l$ then
$$
\begin{pmatrix}
	B & C \\
	0 & 0
\end{pmatrix}=\begin{pmatrix}
	I & C \\
	0 & 0
\end{pmatrix}\begin{pmatrix}
	E_1 & 0 \\
	0 & 1
\end{pmatrix}\cdots \begin{pmatrix}
	E_l & 0 \\
	0 & 1
\end{pmatrix}
$$
is a product of idempotent matrices.

(2)  If $B$ is invertible we can write 
$$
\begin{pmatrix}
	B & C \\
	0 & 0
\end{pmatrix}=\begin{pmatrix}
I & -B^{-1}C \\
0 & 1
\end{pmatrix}\begin{pmatrix}
	B & 0 \\
	0 & 0
\end{pmatrix}\begin{pmatrix}
	I & B^{-1}C \\
	0 & 1
\end{pmatrix}
$$ 
%

(3)  If the column $C$ is right linear combination of the columns of $B$, say $C=BQ$, for some column $Q$ then we have 
$$
\begin{pmatrix}
	B & C \\
	0 & 0
\end{pmatrix}=\begin{pmatrix}
B & BQ \\
0 & 0
\end{pmatrix}=\begin{pmatrix}
B & 0 \\
0 & 0
\end{pmatrix}\begin{pmatrix}
I & Q \\
0 & 0
\end{pmatrix}.
$$
\end{proof}
\vspace{3mm}


\begin{remark}
{\rm
 If $B$ is not a product of idempotents (c.f. when $B$ is invertible) it may still happen that the matrix 
$\begin{pmatrix}
	B & 0 \\
	0 & 0
\end{pmatrix}$ is a product of idempotents.  For example just consider the matrix
$B=\begin{pmatrix}
	X & Y \\
	0 & 0
\end{pmatrix}\in M_2(k[X,Y])$.  The matrix $B$ is not a product of idempotents since the ideal generated by $X$ and $Y$ is not principal.  Note $k[X,Y]$ is local and hence projective free, (cf. Lemma 1, \cite{bb}).  But we can write:
$$
\begin{pmatrix}
	X & Y & 0 \\
	0 & 0 & 0 \\
	0 & 0 & 0
\end{pmatrix}=\begin{pmatrix}
	X & 0 & Y \\
	0 & 0 & 0 \\
	0 & 0 & 0
\end{pmatrix}\begin{pmatrix}
	1 & 0 & 0 \\
	0 & 0 & 0 \\
	0 & 1 & 0
\end{pmatrix}
.$$ 
We note
$$
\begin{pmatrix}
	X & 0 & Y \\
	0 & 0 & 0 \\
	0 & 0 & 0
\end{pmatrix}=\begin{pmatrix}
	1 & 0 & Y \\
	0 & 1 & 0 \\
	0 & 0 & 0
\end{pmatrix}\begin{pmatrix}
	1 & X & 0 \\
	0 & 0 & 0 \\
	0 & 0 & 1
\end{pmatrix}\begin{pmatrix}
	0 & 0 & 0 \\
	0 & 1 & 0 \\
	0 & 0 & 1
\end{pmatrix}\begin{pmatrix}
	1 & 0 & 0 \\
	1 & 0 & 0 \\
	0 & 0 & 1
\end{pmatrix},
$$
 is a product of idempotents.

$$
\begin{pmatrix}
	1 & 0 & 0 \\
	0 & 0 & 0 \\
	0 & 1 & 0
\end{pmatrix}=\begin{pmatrix}
	1 & 0 & 0 \\
	0 & 0 & 0 \\
	0 & 0 & 1
\end{pmatrix}\begin{pmatrix}
	1 & 0 & 0 \\
	0 & 1 & 0 \\
	0 & 1 & 0
\end{pmatrix}
$$

Thus the matrix $\begin{pmatrix}
	B & 0 \\
	0 & 0
\end{pmatrix}$ is a product of  idempotent matrices.
}
\end{remark}

\begin{proposition}
	\label{B-A A of rank r}
	Let $A,B\in M_n(R)$ be such that $A=CD$ where $C\in M_{n\times r}(R)$ and $D\in M_{r\times n}(R)$.  Suppose $B+A$ is a product of idempotent matrices.  Then the matrix
	$$
	\begin{pmatrix}
		B & 0 \\
		0 & 0
	\end{pmatrix}\in M_{n+r}(R)
	$$ 
	is also a product of idempotent matrices
\end{proposition}
\begin{proof}
	Consider
	$$
	\begin{pmatrix}
		B & 0 \\
		0 & 0
	\end{pmatrix}=\begin{pmatrix}
		I_n & -C \\
		0 & 0
	\end{pmatrix}\begin{pmatrix}
		B+A & 0 \\
		D & 0
	\end{pmatrix}
	$$
	Let us write $B+A=E_1\cdots E_l$ where, for any $1\le i \le n$, $E_i^2=E_i$
	we then have 
	$$
	\begin{pmatrix}
		B+A & 0 \\
		D & 0
	\end{pmatrix}=\begin{pmatrix}
		E_1 & 0 \\
		0 & 1
	\end{pmatrix}\cdots \begin{pmatrix}
		E_l & 0 \\
		0 & 1
	\end{pmatrix}\begin{pmatrix}
		I_n & 0 \\
		D & 0
	\end{pmatrix}.
	$$ 
	This gives the proof.
\end{proof}

As a consequence we have the following proposition.

\begin{proposition}
	\label{somme cases}
	Let $B$ be a matrix in $M_n(R)$.   Then the matrix 	$$
	A=\begin{pmatrix}
		B & 0 \\
		0 & 0
	\end{pmatrix}\in M_{n+r}(R)
	$$ 
	is a product of idempotent matrices in each of the following cases:
	\begin{enumerate}
		\item If there exists $1\le i \le n$ such that $(1-e_{ii})B$ (resp. $B(I_n-e_{ii})$) is a product of idempotents.
		\item If $B$ is a product of elementary matrices.
		\item If $B$ is a permutation matrix.
		\item if $B$ is an upper (resp. lower) triangular matrix.
		In particular, if $B$ is a diagonal matrix.
	\end{enumerate}
\end{proposition}
\begin{proof}
It is enough to consider the case when $r=1$.  Indeed if $r>1$ and we assume the result true for $r=1$,  then
the matrix $A'=\begin{pmatrix}
	B & 0 \\
	0 & 0
\end{pmatrix}\in M_{n+1}(R)$ is a product of idempotent matrices.  Hence $A=\begin{pmatrix}
A' & 0 \\
0 & 0
\end{pmatrix}$ is a product of idempotent matrices.
	
	(1) Let $\tilde{B}=(1-e_{ii})B$ be the matrix obtained by replacing the $i^{th}$ row of $B$ by a zero row.  We can then write $\tilde{B}=B - e_iB_i$, where $B_i$ is the $i^{th}$ row of $B$ and $e_i$ is the column with all entries zero except the $i^{th}$ one which is $1$.   Proposition \ref{B-A A of rank r} then applies and gives the case when $r=1$.  
	
	
	(2) It is clear that $A$ is in fact a product of quasi elementray matrices and hence Proposition \ref{Quasi permutation and quasi elementary} shows that $A$ is a product of idempotent matrices.
	
	(3) Since $A$ is clearly a quasi permutation matrix, it is a product of idempotent matrices.
	
	(4) We consider only the case of upper triangular matrices.  The proof proceeds by induction on $n\ge 1$.  If $n=1$ we have the following decomposition
$$
\begin{pmatrix}
	B & 0 \\
	0 & 0
\end{pmatrix}=\begin{pmatrix}
	1 & B \\
	0 & 0
\end{pmatrix}\begin{pmatrix}
	0 & 0 \\
	1 & 1
\end{pmatrix}\begin{pmatrix}
	1 & 0 \\
	0 & 0
\end{pmatrix}
$$
	Assuming the result for upper triangular matrices of size $n-1$, we consider the case of a matrix 
	$A=\begin{pmatrix}
		B & 0 \\
		0 & 0
	\end{pmatrix}$
	where $B\in M_n(R)$ is upper triangular.  We consider, as in (1) above, the matrix $\tilde{B}=B(I_n-e_{nn})$ obtained by replacing the last column of $B$ by zeros.  Since $\tilde{B}$ is upper triangular, the induction hypothesis shows that the matrix $\tilde{B}$ is a product of idempotent and hence (1) above shows that $A$ is also a product of idempotent matrices.
%
\end{proof}

\vspace{3mm}

%

Reacall that $R$ is a separative ring if for any finitely projective right modules $M$ and $N$, $M\oplus N\cong N \oplus N \cong M\oplus M$ implies that $M\cong N$.

\begin{theorem}
	\label{von neumann separative}
	If the ring $R$ is a von Neumann regular separative ring then for any matrix $B\in M_n(R)$ the matrix 
	$$
	\begin{pmatrix}
		B & 0 \\
		0 & 0
	\end{pmatrix}\in M_{n+r}(R), \; r>0
	$$
	is a product of idempotent matrices.
\end{theorem}
\begin{proof}
	If the ring $R$ is supposed to be von Neumann regular and separative then, the Morita invariance of these properties implies that the same is true for any matrix ring over $R$.    Let $B$ be any matrix in  $M_n(R)$.   The main result of \cite{M3} shows that the matrix 
	$$
	A=\begin{pmatrix}
		B & 0 \\
		0 & 0
	\end{pmatrix}\in M_{n+r}(R)
	$$
	is a product of idempotent matrices if and only if the following relations between the annihilators are satisfied:
	$$
	lann(A)S=Srann(A)=S(I_{n+r}-A)S
	$$
	where $S=M_{n+r}(R)$.  Let us prove the first equality.
	We let 
	$$
	X=\begin{pmatrix}
		X_1 & X_2 \\
		X_3 & X_4
	\end{pmatrix}\in lann(A)
	$$
	with $X_1\in M_{n}(R)$ and $X_4\in M_r(R)$ and other matrices are of the appropriate size. 
	We notice that if $X_1=0$ and $X_3=0$ then $X\in lann(A)$.   We then get that $lann(A)$ contains the matrices with the first $n$ columns all  zero.
	Since the right ideal of $S$ generated by these matrices is $S=M_{n+r}(R)$, we get that $lann(A)S=S$.  Similarly we also have 
	$Srann(A)=S$.  Since $A=\begin{pmatrix}
		B & 0 \\
		0 & 0
	\end{pmatrix}$, we can show that the ideal generated by  $I_{n+r}-A$ is the ring $S$ itself.  This concludes the proof.  
\end{proof}

\vspace{3mm}

\begin{remark}
{\rm
	It is an open question whether a von Neumann regular ring is always separative.   
	Let $R$ be a von Neumann regular ring.  If there exists a matrix $B\in M_n(R)$ such that the matrix $\begin{pmatrix}
		B & 0 \\
		0 & 0
	\end{pmatrix}\in M_{n+r}(R)$ cannot be written as a product of idempotent matrices then it follows from Theorem \ref{von neumann separative} that this will provide an example of a regular ring $R$ that is not separative.  If $B$ is invertible, this would also answer the question (3) in \cite{JLS}. 
}	
\end{remark}

\vspace{3mm}


Recall that a ring $R$ is a GE ring, as defined by P.M. Cohn (cf. \cite{C} p. 150), if for any $n>0$, the group $GL_n(R)$ is generated by elementary and diagonal matrices.

It was proved in \cite{AGOP} that a separative exchange ring is a GE ring.  In particular, any separative regular ring is GE.   
We now generalize Theorem 
\ref{von neumann separative} when $B$ is invertible.


\begin{theorem}
	If the ring $R$ is a $GE$ ring then for any invertible matrix $B\in M_n(R)$ the matrix 
	$$
	\begin{pmatrix}
		B & 0 \\
		0 & 0
	\end{pmatrix}
	$$
	is a product of idempotent matrices.
\end{theorem}
\begin{proof}
	If $n=1$ we can consider the decomposition obtained in the proof of Proposition \ref{somme cases} (4).
	
So suppose that $n>1$, and that $R$ is a GE ring.   Since $B$ is invertible, the matrix $B$ is a product of elementary matrices and invertible diagonal matrices.    The conclusion follows easily from  Proposition \ref{somme cases}.
\end{proof}

%
%
%
%


We will now prove that if the matrix $
A=\begin{pmatrix}
			B & 0 \\
			0 & 0
		\end{pmatrix}$
is of size of at least twice the size of $B$ then $A$  is a product of idempotent matrices.

\begin{proposition}
	Let $A\in M_{l}(R)$ be of the form
	$$
	A=\begin{pmatrix}
		B & 0 \\
		0 & 0
	\end{pmatrix}
	$$ where $B\in M_{n}(R)$ and $l\ge 2n$ then $A$ is a product of idempotent matrices.
\end{proposition}
\begin{proof}
	We first remark that, as in the begining of the proof of Proposition \ref{somme cases},  it is enough to consider the case when $l=2n$.   So we assume that $A\in M_{2n}(R)$.  
	Proposition \ref{permutation and elementary matrices acting on a zero row matrix} shows that	 that the given matrix is a product of idempotents if and only if the matrix 
	$$
	A=\begin{pmatrix}
		0 & B \\
		0 & 0
	\end{pmatrix}
	$$
	is a product of idempotents.
	But for this matrix we have the obvious decompositon
	$$
		A=\begin{pmatrix}
		0 & B \\
		0 & 0
	\end{pmatrix}=\begin{pmatrix}
	I & B \\
	0 & 0
\end{pmatrix}\begin{pmatrix}
0 & 0 \\
0 & I
\end{pmatrix}
	$$
	
\end{proof} 


	\section{totally nonnegative matrices}

We conclude this paper by answering in the negative the open question 2 in \cite{JLS}: is it true that a singular totally nonnegative matrix is a product of nonnegative idempotents?  
 The following example shows that this is not true.  Consider the following example of a singular totally nonnegative matrix
$$
A=\begin{pmatrix}
	\alpha & \alpha & 0 & 0 \\
	0 & 0 & 0 & \alpha \\
	\alpha & 0 & 0 & \alpha \\
	0 & \alpha & 0 & 0
\end{pmatrix} \quad \alpha >0
$$
Let us assume that $A=E_1\dots E_n$ where $E_i\ne I_4$ are nonnegative idempotents.   Then 
$E_1A=A$.  Let us write 
$$
E_1=\begin{pmatrix}
	x_1 & y_1 & z_1 & t_1 \\
	x_2 & y_2 & z_2 & t_2 \\
	x_3 & y_3 & z_3 & t_3\\
	x_4 & y_4 & z_4 & t_4
\end{pmatrix}\quad x_i,y_i,z_i,t_i >0 \quad {\rm for} \; 1\le i \le 4
$$
Comparing the first rows of $E_1$ and $E_1A$, we get $x_1+z_1=1,\; x_1+t_1=1,\; y_1+z_1=0$.  This gives $x_1=1,\;y_1=z_1=t_1=0$.  Continuing in this way with other rows leads to $E_1=I_4$, a contradiction.

\end{document}